\documentclass[12pt]{amsart}
\usepackage{amsmath, amssymb, epsfig}
\usepackage{bm}
\usepackage{mathrsfs}
\usepackage{color}
\usepackage{alg}
\usepackage{fullpage}
\usepackage{subfigure}

\makeatletter
\def\algspacing{\alg@unmargin}
\makeatother

\newlength{\algorithmwidth}
\theoremstyle{plain}
\newtheorem{theorem}{Theorem}[section]

\newtheorem{lemma}[theorem]{Lemma}

\theoremstyle{definition}

\theoremstyle{remark}
\newtheorem*{remark}{Remark}

\numberwithin{equation}{section}

\newcommand{\<}{\left\langle}
\renewcommand{\>}{\right\rangle}
\newcommand{\bigO}{\mathrm{O}}

\newcommand{\defby}{\overset{\mathrm{\scriptscriptstyle{def}}}{=}}

\def \E {\mathbb{E}}

\newcommand{\vct}[1]{\bm{#1}}
\newcommand{\mtx}[1]{\bm{#1}}

\begin{document}
\bibliographystyle{plain}
\setlength{\parindent}{0in}
\parskip 7.2pt

\title[]{Two-subspace Projection Method for Coherent Overdetermined Systems}
\author{Deanna Needell and Rachel Ward}
\date{\today}

\begin{abstract}
We present a Projection onto Convex Sets (POCS) type algorithm for solving systems of linear equations.  POCS methods have found many applications ranging from computer tomography to digital signal and image processing.  The Kaczmarz method is one of the most popular solvers for overdetermined systems of linear equations due to its speed and simplicity.  Here we introduce and analyze an extension of the Kaczmarz method that iteratively projects the estimate onto a solution space given by two randomly selected rows.  We show that this projection algorithm provides exponential convergence to the solution in expectation.  The convergence rate improves upon that of the standard randomized Kaczmarz method when the system has correlated rows.  Experimental results confirm that in this case our method significantly outperforms the randomized Kaczmarz method.

\end{abstract}
\maketitle

\section{Introduction}\label{sec:intro}
We consider a consistent system of linear equations of the form
$$
\mtx{A}\vct{x} = \vct{b},
$$
where $\vct{b}\in\mathbb{C}^m$ and $\mtx{A}\in\mathbb{C}^{m\times n}$ is a full-rank $m\times n$ matrix that is overdetermined, having more rows than columns ($m\geq n$).  When the number of rows of $\mtx{A}$ is large, it is far too costly to invert the matrix to solve for $\vct{x}$, so one may utilize an iterative solver such as the Projection onto Convex Sets (POCS) method, used in many applications of signal and image processing~\cite{CFMSS92,SS87:Applications}.  The Kaczmarz method is often preferred, iteratively cycling through the rows of $\mtx{A}$ and orthogonally projecting the estimate onto the solution space given by each row~\cite{K37:Angena}.  Precisely, let us denote by $\vct{a_1}$, $\vct{a_2}$, $\ldots$, $\vct{a_m}$ the rows of $\mtx{A}$ and $b_1$, $b_2$, $\ldots$, $b_m$ the coordinates of $\vct{b}$.  We assume each pair of rows is linear independent, and for simplicity, we will assume throughout that the matrix $\mtx{A}$ is \textit{standardized}, meaning that each of its rows has unit Euclidean norm; generalizations from this case will be straightforward.  Given some trivial initial estimate $\vct{x_0}$, the Kaczmarz method cycles through the rows of $\mtx{A}$ and in the $k$th iteration projects the previous estimate $\vct{x_k}$ onto the solution hyperplane of $\langle \vct{a_i}, \vct{x}\rangle = b_i$ 
where $i = k$ mod $m$, 
$$
\vct{x_{k+1}} = \vct{x_k} + (b_i - \langle \vct{a_i}, \vct{x_{k}} \rangle)\vct{a_i}.
$$

Theoretical results about the rate of convergence of the Kaczmarz method have been difficult to obtain, and most are based on quantities which are themselves hard to compute~\cite{DH97:Therate,G05:Onthe}. Even more importantly, the method as we have just described depends heavily on the ordering of the rows of $\mtx{A}$.  A malicious or unlucky ordering may therefore lead to extremely slow convergence.  To overcome this, one can select the rows of $\mtx{A}$ in a \emph{random} fashion rather than cyclically~\cite{HM93:Algebraic,N86:TheMath}.  Strohmer and Vershynin analyzed a randomized version of the Kaczmarz method that in each iteration selects a row of $\mtx{A}$ with probability proportional to the square of its Euclidean norm~\cite{SV09:Arand,SV06:Arandom}.  Thus in the standardized case we consider, a row of $\mtx{A}$ is chosen uniformly at random.  This randomized Kaczmarz method is described by the following pseudocode.

\begin{algorithm}[thb]
\caption{Randomized Kaczmarz}
	\label{alg:rk}
\centering \fbox{
\begin{minipage}{.99\textwidth} 
\vspace{4pt}
\alginout{Standardized matrix $\mtx{A}$, vector $\vct{b}$}
{An estimation $\vct{x_k}$ of the unique solution $\vct{x}$ to $\mtx{A}\vct{x} = \vct{b}$
}
\vspace{8pt}\hrule\vspace{8pt}

\begin{algtab*}
Set $\vct{x_0}$.
	 	\hfill \{ Trivial initial approximation \} \\		 
$k \leftarrow 0$ \\

\algrepeat
	$k \leftarrow k + 1$ \\
	Select $ r \in \{1, 2, \ldots, n\}$
		\hfill \{ Randomly select a row of $\mtx{A}$ \} \\
		Set $ \vct{x_k} \leftarrow \vct{x_{k-1}} + (b_r - \langle \vct{a_r},\vct{x_{k-1}} \rangle)\vct{a_r}$
		\hfill \{ Perform projection \} \\
	
\end{algtab*}
\vspace{-10pt}
\end{minipage}}
\end{algorithm}

Note that this method as stated selects each row \textit{with replacement}, see~\cite{Valley} for a discussion on the differences in performance when selecting with and without replacement.  Strohmer and Vershynin show that this method exhibits exponential convergence in expectation~\cite{SV09:Arand,SV06:Arandom},
\begin{equation}\label{SV}
\E \|\vct{x_k} - \vct{x}\|_2^2 \leq \left( 1 - \frac{1}{R}\right)^k\|\vct{x_0} - \vct{x}\|_2^2,\quad\text{where}\quad R \defby \|\mtx{A}\|_F^2\|\mtx{A}^{-1}\|^2.
\end{equation}
Here and throughout, $\|\cdot\|_2$ denotes the vector Euclidean norm, $\| \cdot \|$ denotes the matrix spectral norm, $\|\cdot\|_F$ denotes the matrix Frobenius norm, and the inverse $\|\mtx{A}^{-1}\| = \inf\{M : M\|\mtx{A}\vct{x}\|_2 \geq \|\vct{x}\|_2 \text{ for all }\vct{x}\}$ is well-defined since $\mtx{A}$ is full-rank.  This bound shows that when $\mtx{A}$ is well conditioned, the randomized Kaczmarz method will converge exponentially to the solution in just $\bigO(n)$ iterations (see Section 2.1 of~\cite{SV09:Arand} for details).  The cost of each iteration is the cost of a single projection and takes $\bigO(n)$ time, so the total runtime is just $\bigO(n^2)$.  This is superior to Gaussian elimination which takes $\bigO(mn^2)$ time, especially for very large systems.  The randomized Kaczmarz method even substantially outperforms the well-known conjugate gradient method in many cases~\cite{SV09:Arand}.  

Leventhal and Lewis show that for certain probability distributions, the expected rate of convergence can be bounded in terms of other natural linear-algebraic quantities.  They propose generalizations to other convex systems~\cite{leventhal2010randomized}.
Recently, Chen and Powell proved that for certain classes of random matrices $\mtx{A}$, the randomized Kaczmarz method convergences exponentially to the solution not only in expectation but also almost surely~\cite{CP12:almost}.    

In the presence of noise, one considers the possibly inconsistent system $\mtx{A}\vct{x} + \vct{w} \approx \vct{b}$ for some error vector $\vct{w}$.  In this case the randomized Kaczmarz method converges exponentially fast to the solution within an error threshold~\cite{N10:rknoisy},
$$
\E\|\vct{x_k} - \vct{x}\|_2 \leq \left( 1 - \frac{1}{R}\right)^{k/2}\|\vct{x_0} - \vct{x}\|_2 + \sqrt{R}\|\vct{w}\|_{\infty},
$$
where $R$ the the scaled condition number as in \eqref{SV} and $\|\cdot\|_{\infty}$ denotes the largest entry in magnitude of its argument.  This error is sharp in general~\cite{N10:rknoisy}.  Modified Kaczmarz algorithms can also be used to solve the least squares version of this problem, see for example~\cite{drineas2007faster,ENP10:semi,HN90:Onthe,censor1983strong} and the references therein.

\subsection{Coherent systems}  

Although the convergence results for the randomized Kaczmarz method hold for any consistent system, the factor $\frac{1}{R}$ in the convergence rate may be quite small for matrices with many correlated rows.  
Consider for example the reconstruction of a bandlimited function from nonuniformly spaced samples, as often arises in geophysics as it can be physically challenging to take uniform samples.  Expressed as a system of linear equations, the sampling points form the rows of a matrix $\mtx{A}$; for points that are close together, the corresponding rows will be highly correlated.  

To be precise, we examine the pairwise \textit{coherence} of a standardized matrix $\mtx{A}$ by defining the quantities
\begin{equation}\label{mus}
\Delta = \Delta(\mtx{A}) = \max_{j\ne k}|\langle \vct{a_j}, \vct{a_k}\rangle| \quad{and}\quad
\delta = \delta(\mtx{A}) = \min_{j\ne k}|\langle \vct{a_j}, \vct{a_k}\rangle|.
\end{equation}

\begin{remark}These quantities measure how correlated the rows of the matrix $\mtx{A}$ are.  We point out that this notion of coherence coincides with that of signal processing terminology and is different than the alternative definition which measures the correlation between singular vectors and the canonical vectors.  The notion of coherence used here simply gives a measure of pairwise row correlation. Analysis using the notion of coherence for singular vectors may also lead to improved convergence rates for these methods, and we leave this as future work.
\end{remark}

Note also that because $\mtx{A}$ is standardized, $0 \leq \delta \leq \Delta \leq 1$. 
It is clear that when $\mtx{A}$ has high coherence parameters, $\|\mtx{A}^{-1}\|$ is very small and thus the factor $R$ in~\eqref{SV} is also small, leading to a weak bound on the convergence.  Indeed, when the matrix has highly correlated rows, the angles between successive orthogonal projections are small and convergence is stunted.   We can explore a wider range of orthogonal directions by looking towards solution hyperplanes spanned by  \emph{pairs} of rows of $\mtx{A}$.  We thus propose a modification to the randomized Kaczmarz method where each iteration performs an orthogonal projection onto a two-dimensional subspace spanned by a randomly-selected pair of rows.  We point out that the idea of projecting in each iteration onto a subspace obtained from multiple rows rather than a single row has been previously investigated numerically, see e.g.~\cite{FS95,CFMSS92}.
 
With this as our goal, a single iteration of the modified algorithm will consist of the following steps.  Let ${\vct{x_k}}$ denote the current estimation in the $k$th iteration.

\begin{itemize}
\item Select two distinct rows $\vct{a_r}$ and $\vct{a_s}$ of the matrix $\mtx{A}$ at random
\item Compute the translation parameter $\varepsilon$
\item Perform an intermediate projection: $\vct{y} \leftarrow \vct{x_k} + \varepsilon(b_r - \langle \vct{x_k}, \vct{a_r}\rangle)\vct{a_r}$
\item Perform the final projection to update the estimation: $\vct{x_{k+1}} \leftarrow \vct{y} + (b_s - \langle \vct{y}, \vct{a_s}\rangle)\vct{a_s}$
\end{itemize}

In general, the optimal choice of $\varepsilon$ at each iteration of the two-step procedure corresponds to subtracting from $\vct{x_k}$ its orthogonal projection onto the solution space $\{\vct{x}: \langle \vct{a_r}, \vct{x}\rangle = b_r \text{ and } \langle \vct{a_s}, \vct{x}\rangle = b_s\}$, which motivates the name two-subspace Kaczmarz method. 
By \emph{optimal choice} of $\varepsilon$, we mean the value $\varepsilon_{opt}$ minimizing the residual $\|\vct{x} - \vct{x_{k+1}}\|_2^2$.   Expanded, this reads  
$$
\|\vct{x} - \vct{x_{k+1}}\|_2^2 = \|\varepsilon(b_r - \<\vct{x_k}, \vct{a_r}\>)(\vct{a_r} - \<\vct{a_s}, \vct{a_r}\>\vct{a_s}) + \vct{x_k} - \vct{x} + (b_s - \<\vct{x_k}, \vct{a_s}\>)\vct{a_s}\|_2^2.
$$

Using that the minimizer of $\|\gamma \vct{w} + \vct{z}\|_2^2$ is $\gamma = -\frac{\<\vct{w},\vct{z}\>}{\|\vct{w}\|_2^2}$, we see that 
\begin{equation*}
\varepsilon_{opt} = \frac{-\<\vct{a_r} - \<\vct{a_s}, \vct{a_r}\>\vct{a_s} , \vct{x_k} - \vct{x} + (b_s - \<\vct{x_k},\vct{a_s}\>)\vct{a_s}\>}{(b_r - \<\vct{x_k}, \vct{a_r}\>)\|\vct{a_r} - \<\vct{a_s}, \vct{a_r}\>\vct{a_s}\|_2^2}.
\end{equation*}
Note that the unknown vector $\vct{x}$ appears in this expression only through its observable inner products, and so $\varepsilon_{opt}$ is computable.  
After some algebra, one finds that the two-step procedure with this choice of $\varepsilon_{opt}$ can be re-written in the following numerically stable formulation.

\begin{algorithm}[thb]
\caption{Two-subspace Kaczmarz}
	\label{alg:r2k}
\centering \fbox{
\begin{minipage}{.99\textwidth} 
\vspace{4pt}
\alginout{Matrix $\mtx{A}$, vector $\vct{b}$}
{An estimation $\vct{x_k}$ of the unique solution $\vct{x}$ to $\mtx{A}\vct{x} = \vct{b}$
}
\vspace{8pt}\hrule\vspace{8pt}

\begin{algtab*}
Set $\vct{x_0}$.
	 	\hfill \{ Trivial initial approximation \} \\		 
$k \leftarrow 0$ \\

\algrepeat
	$k \leftarrow k + 1$ \\
	Select $ r, s \in \{1, 2, \ldots, n\}$
		\hfill \{ Select two distinct rows of $\mtx{A}$ uniformly at random \} \\
		Set $ \mu_k \leftarrow \langle \vct{a_{r}}, \vct{a_{s}}\rangle$
		\hfill \{ Compute correlation \} \\
		Set $ \vct{y_k} \leftarrow \vct{x_{k-1}} + (b_s - \langle \vct{x_{k-1}}, \vct{a_s}\rangle)\vct{a_s}$
		\hfill \{ Perform intermediate projection \} \\
		Set $ \vct{v_k} \leftarrow \frac{\vct{a_r} - \mu_k \vct{a_s}}{\sqrt{1-|\mu_k|^2}}$
		\hfill \{ Compute vector orthogonal to $\vct{a_s}$ in direction of $\vct{a_r}$ \} \\
		Set $ \beta_k \leftarrow \frac{b_r - b_s\mu_k}{\sqrt{1-|\mu_k|^2}}$
		\hfill \{ Compute corresponding measurement \}\\
	$\vct{x_k} \leftarrow \vct{y_k} + (\beta_k - \langle \vct{y_k}, \vct{v_k}\rangle)\vct{v_k}$
		\hfill \{ Perform projection \} \\	
	
\end{algtab*}
\vspace{-10pt}
\end{minipage}}
\end{algorithm}

We note that by the assumption that each pair of rows is linearly independent, we have $|\mu_k| \ne 1$ for all $k$ so that division is always well-defined.    
Our main result shows that the two-subspace Kaczmarz algorithm provides the same exponential convergence rate as the standard method in general, and substantially improved convergence when the rows of $\mtx{A}$ are coherent.  Figure \ref{fig:A} plots two iterations of the one-subspace random Kaczmarz algorithm and compares this to a single iteration of the two-subspace Kaczmarz algorithm.

\begin{figure}[ht]
\centering
\subfigure[]{
   \includegraphics[scale=0.5] {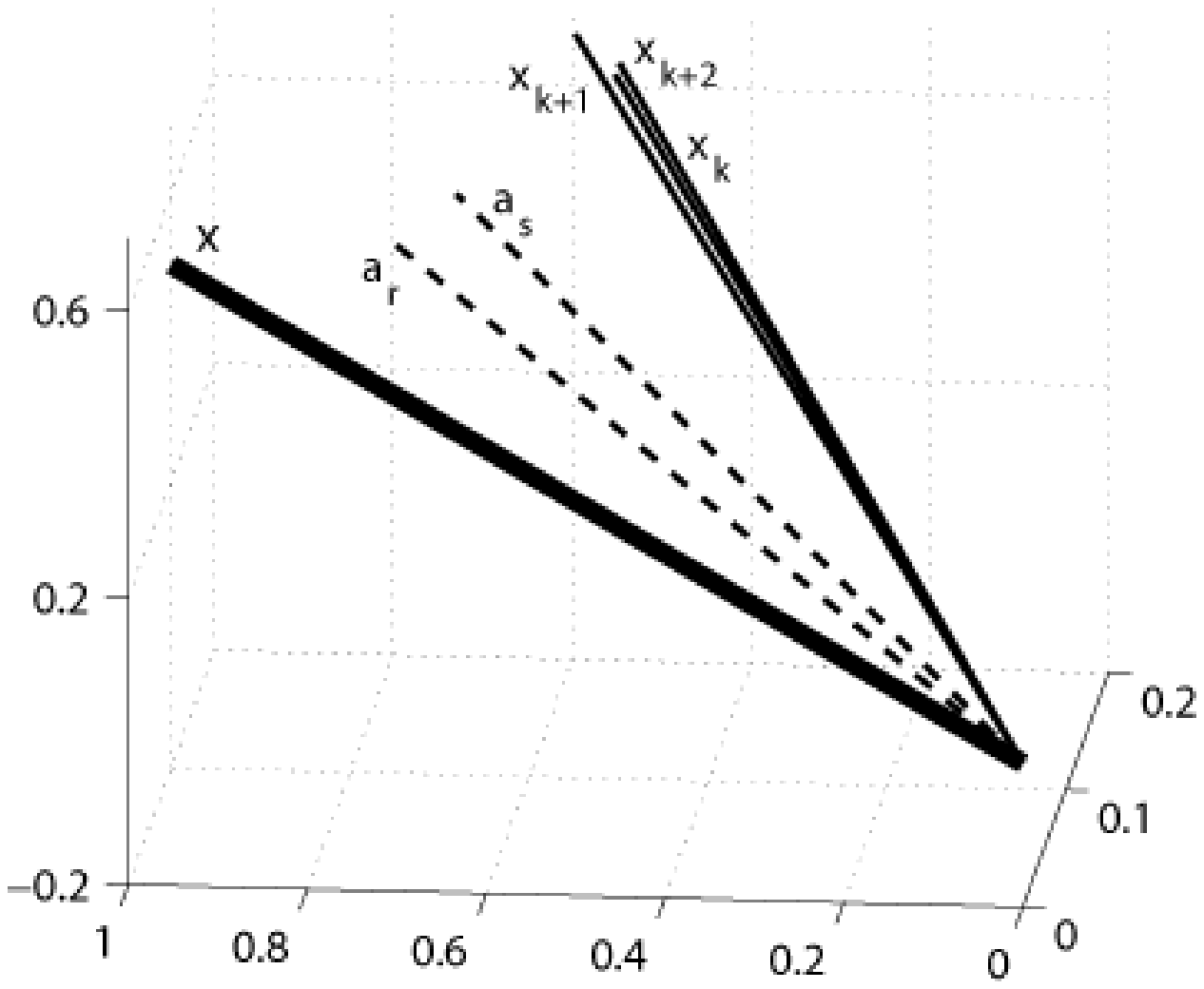}
 }
 \subfigure[]{
   \includegraphics[scale=0.5] {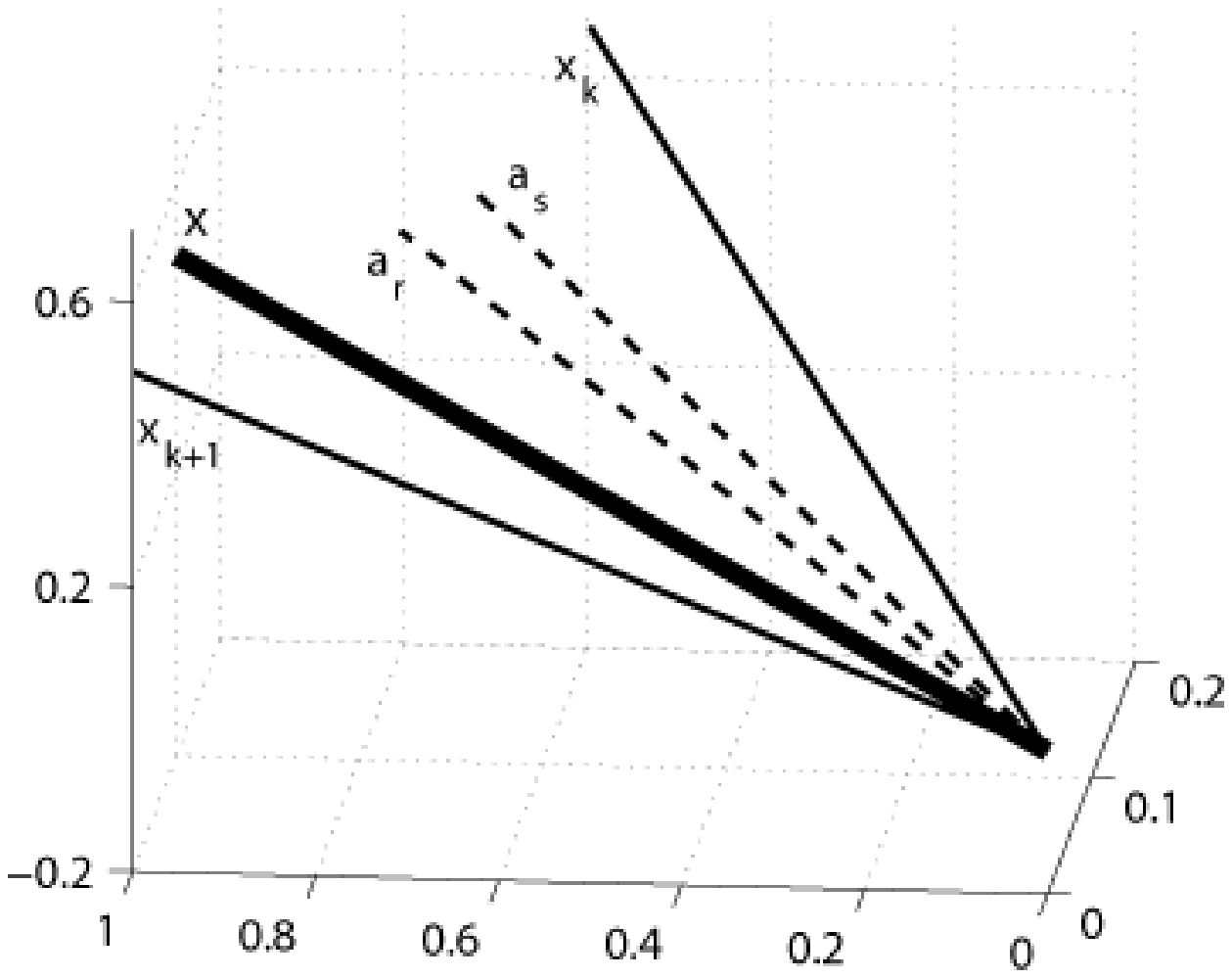}
 }
\caption{For coherent systems,  the one-subspace randomized Kaczmarz algorithm (a) converges more slowly than the two-subspace Kaczmarz algorithm (b).} \label{fig:A}
\end{figure}

\begin{theorem}\label{thm:main}
Let $\mtx{A}$ be a full-rank standardized matrix with $n$ columns and  $m > n$ rows and suppose $\mtx{A}\vct{x} = \vct{b}$.  Let $\vct{x_k}$ denote the estimation to the solution $\vct{x}$ in the $k$th iteration of the two-subspace Kaczmarz method.  Then
$$
\mathbb{E}\|\vct{x} - \vct{x_k}\|_2^2 \leq \left(\left(1 - \frac{1}{R}\right)^2 - \frac{D}{R}\right)^k\|\vct{x} - \vct{x}_{0}\|_2^2 ,
$$

where $D = \min\Big\{ \frac{\delta^2(1-\delta)}{1+\delta}, \frac{\Delta^2(1-\Delta)}{1+\Delta} \Big\} $, $\Delta$ and $\delta$ are the coherence parameters~\eqref{mus}, and $R = \| \mtx{A} \|_{F}^2 \|\mtx{A}^{-1}\|^2$ denotes the scaled condition number.
\end{theorem}

\begin{remarks}
{\bfseries 1. }When $\Delta = 1$ or $\delta = 0$ we recover the same convergence rate as provided for the standard Kaczmarz method~\eqref{SV} since the two-subspace method utilizes two projections per iteration. 

{\bfseries 2. }The bound presented in Theorem~\ref{thm:main} is a pessimistic bound.  Even when $\Delta = 1$ or $\delta = 0$, the two-subspace method improves on the standard method if any rows of $\mtx{A}$ are highly correlated (but not equal). This is evident from the proof of Theorem~\ref{thm:main} in Section~\ref{sec:proofs} via Lemma~\ref{lem:main} but we present the bound for simplicity.  Under other assumptions on the matrix $\mtx{A}$, improvements can be made to the convergence bound of Theorem~\ref{thm:main}.  For example, if one assumes that the correlations between the rows are non-negative, one obtains the bound
$$
\mathbb{E}\|\vct{x} - \vct{x_k}\|_2^2 \leq \left(\left(1 - \frac{1}{R}\right)^2 - \frac{D}{R} - \frac{E}{Q}\right)^k\|\vct{x} - \vct{x}_{0}\|_2^2 ,
$$
where $E = 4\delta^3$ and $Q = \|\mtx{\Omega}^{-1}\|^2\|\mtx{\Omega}\|_F^2$ is the scaled condition number of the $m^2\times n$ matrix $\mtx{\Omega}$ whose rows consist of normalized row differences from $\mtx{A}$, $\vct{a_j} - \vct{a_i}$.  See~\cite{NW12:2srkTech} for details and the proof of this result.

{\bfseries 3. }Theorem~\ref{thm:main} yields a simple bound on the expected runtime of the two-subspace randomized Kaczmarz method.  To achieve accuracy $\varepsilon$, meaning
$$
\mathbb{E}\|\vct{x_k} - \vct{x}\|_2^2 \leq \varepsilon^2\|\vct{x_0} - \vct{x}\|_2^2,
$$
one asks that
$$
\mathbb{E}(k) \leq \frac{2\log\varepsilon}{\log\left((1-\frac{1}{R})^2 - \frac{D}{R}\right)}.
$$
If $\mtx{A}$ is well-conditioned then $R = \bigO(n)$ and we thus require that 
$$
k = \bigO\left(\frac{2n}{2+D - \frac{1}{n}}\right).
$$
Since each iteration requires $\bigO(n)$ time, for large enough $n$ this again yields a total runtime of $\bigO(n^2)$ as in the standard randomized Kaczmarz case~\cite{SV09:Arand}, but with an improvement in the constant factors.

{\bfseries 4. }When the rows of $\mtx{A}$ have arbitrary norms, one may simply select pairs of rows uniformly at random, normalize prior to performing the projections, and obtain the result of Theorem~\ref{thm:main} in terms of the standardized matrix.  One obtains an alternative bound by selecting pairs of distinct rows $\vct{a_r}$ and $\vct{a_s}$ with probability proportional to the product $\|\vct{a_r}\|_2^2\|\vct{a_s}\|_2^2$, following the strategy of Strohmer and Vershynin~\cite{SV09:Arand} in the standard randomized Kaczmarz algorithm.  Defining the normalized variables $\vct{\widetilde{a}_r} = \vct{{a}_r}/\|\vct{a_r}\|_2$  and $\widetilde{b}_r = b_r /\|\vct{a_r}\|_2$, the algorithm proceeds as before with these substitutions in place.  We define the coherence parameters~\eqref{mus} in terms of the normalized rows, and we define a new matrix $\hat{\mtx{A}} := \mtx{DA}$, where $\mtx{D}$ is a diagonal matrix with entries $\mtx{D}_{jj} = \left(\|\mtx{A}\|_F^2 - \|\vct{a_j}\|_2^2\right)^{1/2}$.  Then one follows the proof of Theorem~\ref{thm:main} to obtain the analogous convergence bound in the non-standardized case,

$$
\mathbb{E}\|\vct{x} - \vct{x_k}\|_2^2 \leq \left(\left(1 - \frac{1}{\hat{R}}\right)^2 - \frac{D}{\hat{R}}\right)^k\|\vct{x} - \vct{x}_{0}\|_2^2 ,
$$
where $D$ is as in Theorem~\ref{thm:main}, and $\hat{R} = \| \mtx{\hat{A}} \|_{F}^2 \|\mtx{\hat{A}}^{-1}\|^2$ denotes the scaled condition number of $\hat{\mtx{A}}$. 

\end{remarks}

Figure~\ref{fig:D} shows the value of $D$ of Theorem \ref{thm:main} for various values of $\Delta$ and $\delta$.  This demonstrates that the improvement factor $D$ is maximized when $\delta = \Delta = \frac{\sqrt{5}-1}{2} \approx 0.62$, giving a value of $D \approx 0.1$.  

  \begin{figure}[h!]
\begin{center}
\includegraphics[width=3in]{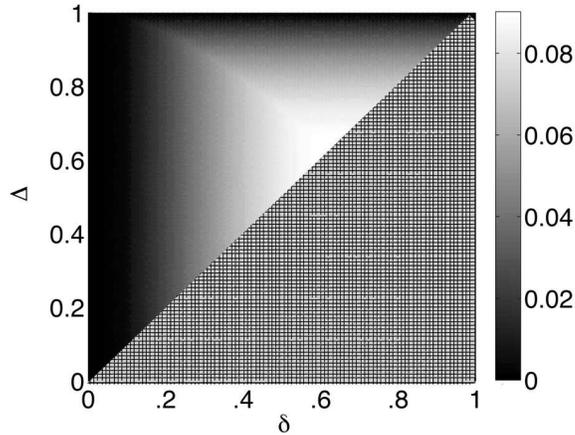}    
\end{center}
\caption{A plot of the improved convergence factor $D$ as a function of the coherence parameters $\delta$ and $\Delta \geq \delta$. }\label{fig:D}
\end{figure}

\subsection{Organization} Next in Section~\ref{sec:numerics} we present some numerical results demonstrating the improvements offered by the two-subspace randomized Kaczmarz method.  We then prove our main result, Theorem~\ref{thm:main} in Section~\ref{sec:proofs}.  We end with a brief discussion in Section~\ref{sec:conclusion}.

\section{Numerical Results}\label{sec:numerics}

In this section we perform several experiments to compare the convergence rate of the two-subspace randomized Kaczmarz with that of the standard randomized Kaczmarz method.  As discussed, both methods exhibit exponential convergence in expectation, but when the rows of the matrix $\mtx{A}$ are coherent, the two-subspace method exhibits much faster convergence. 

To test these methods, we construct various types of $300\times 100$ matrices $\mtx{A}$.  To acquire a range of $\delta$ and $\Delta$, we set the entries of $\mtx{A}$ to be independent identically distributed uniform random variables on some interval $[c, 1]$.  Changing the value of $c$ will appropriately change the values of $\delta$ and $\Delta$.  Note that there is nothing special about this interval, other intervals (both negative and positive or both) of varying widths yield the same results.  For each matrix construction, both the randomized Kaczmarz and two-subspace randomized methods are run with the same fixed initial (randomly selected) estimate and fixed matrix.  The estimation errors for each method are computed at each iteration and averaged over $40$ trials.  The heavy lines depict the average error over these trials, and the shaded region describes the minimum and maximum errors.  Since each iteration of the two-subspace method utilizes two rows of the matrix $\mtx{A}$, we will equate a single iteration of the standard method with two iterations in Algorithm~\ref{alg:rk} for fair comparison.

Figure~\ref{fig:highCoh} demonstrates the regime where the two-subspace method offers the most improvement over the standard method.  Here the matrix $A$ has highly coherent rows, with $\delta \approx \Delta$. 

  \begin{figure}[h!]
\begin{center}
\includegraphics[width=3in]{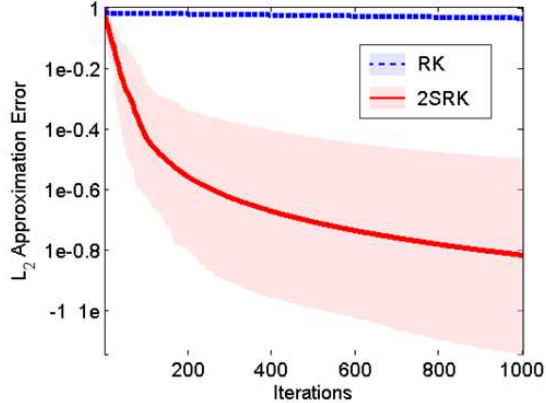}    
\end{center}
\caption{A log-linear plot of the error per iteration for the randomized Kaczmarz (RK) and two-subspace RK (2SRK) methods.  Matrix $\mtx{A}$ has highly coherent rows, with entries uniformly distributed on $[0.9, 1]$ yielding $\delta = 0.998$ and $\Delta = 0.999$. }\label{fig:highCoh}
\end{figure}

Our result Theorem~\ref{thm:main} suggests that as $\delta$ becomes smaller the two-subspace method should offer less and less improvements over the standard method.  When $\delta=0$ the convergence rate bound of Theorem~\ref{thm:main} is precisely the same as that of the standard method~\eqref{SV}.  Indeed, we see this precise behavior as is depicted in Figure~\ref{fig:others}. 

\begin{figure}[h!]
\begin{center}
$\begin{array}{c@{\hspace{.1in}}c}
(a)\includegraphics[width=2.5in]{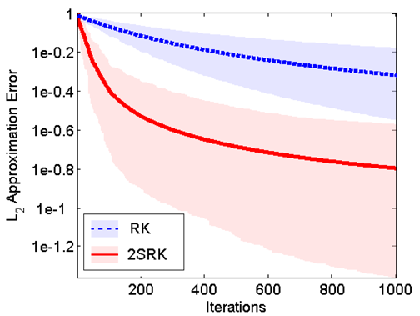} &
(b)\includegraphics[width=2.5in]{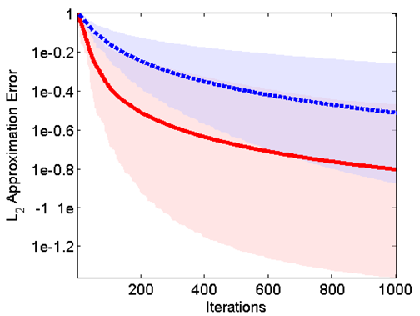} \\
(c)\includegraphics[width=2.5in]{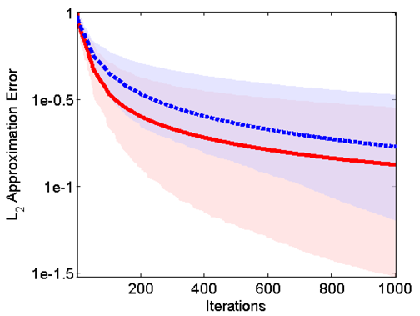} &
(d)\includegraphics[width=2.5in]{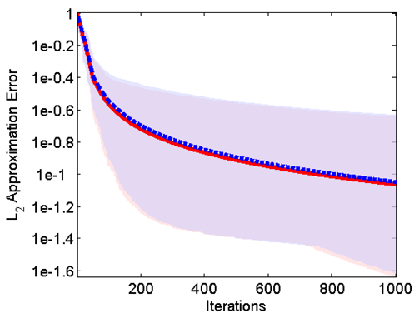}\\
\end{array}$    
\end{center}
\caption{A log-linear plot of the error per iteration for the randomized Kaczmarz (RK) and two-subspace RK (2SRK) methods.  Matrix $\mtx{A}$ has entries uniformly distributed on $[c, 1]$ with coherence parameters (a) $\delta = 0.937$ and $\Delta = 0.986$ ($c=0.5$), (b) $\delta = 0.760$ and $\Delta = 0.954$ ($c=0.2$), (c) $\delta = 0.394$ and $\Delta = 0.870$ ($c=-0.1$), and (d) $\delta = 0$ and $\Delta = 0.740$ ($c=-0.5$). }\label{fig:others}
\end{figure}

\section{Main Results} \label{sec:proofs}

We now present the proof of Theorem~\ref{thm:main}.  We first derive a bound for the expected progress made in a single iteration.  Since the two row indices are chosen independently at each iteration, we will be able to apply the bound recursively to obtain the desired overall expected convergence rate.  

Our first lemma shows that the expected estimation error in a single iteration of the two-subspace Kaczmarz method is decreased by a factor strictly less than that of the standard randomized method.

\begin{lemma}\label{lem:main}
 Let $\vct{x_{k}}$ denote the estimation to the solution of $\mtx{A}\vct{x} = \vct{b}$ in the $k$th iteration of the two-subspace Kaczmarz method.  Denote the rows of $\mtx{A}$ by $\vct{a}_1, \vct{a}_2, \ldots \vct{a}_m$.  Then we have the following bound,
$$
\mathbb{E}\|\vct{x} - \vct{x_{k}}\|_2^2 \leq \left(1 - \frac{1}{R}\right)^2\|\vct{x} - \vct{x_{k-1}}\|_2^2 - \frac{1}{m^2-m}\sum_{r < s} C_{r,s}^2\left(\langle \vct{x} - \vct{x_{k-1}}, \vct{a_{r}}\rangle^2 + \langle \vct{x} - \vct{x_{k-1}}, \vct{a_{s}}\rangle^2\right),
$$

where $C_{r,s} = \frac{|\mu_{r,s}|-\mu_{r,s}^2}{\sqrt{1-\mu_{r,s}^2}}$, $\mu_{r,s} = \<\vct{a_r}, \vct{a_s}\>$, and $R = \|\mtx{A}^{-1}\|^2\|\mtx{A}\|_F^2$ denotes the scaled condition number.

\end{lemma}

\begin{proof}
  We fix an iteration $k$ and for convenience refer to $\vct{v}_k$, $\mu_k$, and $\vct{y}_k$ as $\vct{v}$, $\mu$, and $\vct{y}$, respectively.  We will also denote $\gamma = \langle \vct{a_{r}}, \vct{v}\rangle$.  
  
 First, observe that by the definitions of $\vct{v}$ and $\vct{x_{k}}$ we have
$$
\vct{x_{k}} = \vct{x_{k-1}} + \langle \vct{x} - \vct{x_{k-1}}, \vct{a_{s}}\rangle \vct{a_{s}} + \langle \vct{x} - \vct{x_{k-1}}, \vct{v}\rangle \vct{v}.
$$
Since $\vct{a_{s}}$ and $\vct{v}$ are orthonormal, this gives the estimate
\begin{equation}\label{eq:new}
\|\vct{x} - \vct{x_{k}}\|_2^2 = \|\vct{x} - \vct{x_{k-1}}\|_2^2 - |\langle \vct{x} - \vct{x_{k-1}}, \vct{a_{s}}\rangle|^2 - |\langle \vct{x} - \vct{x_{k-1}}, \vct{v}\rangle|^2
\end{equation}

We wish to compare this error with the error from the standard randomized Kaczmarz method.  Since we utilize two rows per iteration in the two-subspace Kaczmarz method, we compare its error with the error from two iterations of the standard method.   Let $\vct{z}$ and $\vct{z'}$ be two subsequent estimates in the standard method following the estimate  $\vct{x_{k-1}}$, and assume $\vct{z} \neq \vct{z'}$.  That is,
\begin{equation}\label{eq:z}
\vct{z} = \vct{x_{k-1}} +  (b_r - \langle \vct{x_{k-1}}, \vct{a_{r}}\rangle)\vct{a_{r}} \quad\text{and}\quad
\vct{z'} = \vct{z} +  (b_s - \langle \vct{z}, \vct{a_{s}}\rangle)\vct{a_{s}}.
\end{equation} 

Recalling the definitions of $\vct{v}$, $\mu$ and $\gamma$, we have
\begin{equation}\label{eq:a}
\vct{a_{r}} = \mu\vct{a_{s}} + \gamma \vct{v}\quad\text{with}\quad \mu^2 + \gamma^2 = 1.
\end{equation}

Substituting this into~\eqref{eq:z} yields
$$
\vct{z} = \vct{x_{k-1}} + \mu\langle \vct{x} - \vct{x_{k-1}}, \vct{a_{r}}\rangle \vct{a_{s}} + \gamma\langle\vct{x} - \vct{x_{k-1}}, \vct{a_{r}}\rangle\vct{v}.
$$

Now substituting this into~\eqref{eq:z} and taking the orthogonality of $\vct{a_{s}}$ and $\vct{v}$ into account,
$$
\vct{z}' = \vct{x_{k-1}} + \langle \vct{x} - \vct{x_{k-1}}, \vct{a_{s}}\rangle \vct{a_{s}} + \gamma\langle\vct{x} - \vct{x_{k-1}}, \vct{a_{r}}\rangle\vct{v}.
$$

For convenience, let $\vct{e_{k-1}} = \vct{x} - \vct{x_{k-1}}$ denote the error in the $(k-1)$st iteration of two-subspace Kaczmarz. Then we have
\begin{align*}
\|\vct{x} - \vct{z'}\|_2^2 
&= \|\vct{e_{k-1}} - \langle \vct{e_{k-1}}, \vct{a_{s}}\rangle \vct{a_{s}} - \gamma\langle\vct{e_{k-1}}, \vct{a_{r}}\rangle\vct{v}\|_2^2\\
&= \|\vct{e_{k-1}} - \langle \vct{e_{k-1}}, \vct{a_{s}}\rangle \vct{a_{s}} - \langle \vct{e_{k-1}}, \vct{v}\rangle\vct{v} - (\gamma\langle\vct{e_{k-1}}, \vct{a_{r}}\rangle - \langle \vct{e_{k-1}}, \vct{v}\rangle)\vct{v}\|_2^2\\
&= \|\vct{e_{k-1}}\|_2^2 - |\langle \vct{e_{k-1}}, \vct{a_{s}}\rangle|^2 - |\langle \vct{e_{k-1}}, \vct{v}\rangle|^2 + |\gamma\langle\vct{e_{k-1}}, \vct{a_{r}}\rangle - \langle \vct{e_{k-1}}, \vct{v}\rangle|^2.
\end{align*}
The third equality follows from the orthonormality of $\vct{a_{s}}$ and $\vct{v}$.
We now expand the last term,
\begin{align*}
|\gamma\langle\vct{e_{k-1}}, \vct{a_r}\rangle - \langle \vct{e_{k-1}}, \vct{v}\rangle|^2 &=
|\gamma\langle\vct{e_{k-1}}, \mu\vct{a_{s}} + \gamma\vct{v}\rangle - \langle \vct{e_{k-1}}, \vct{v}\rangle|^2\\
&=  |\gamma^2\langle\vct{e_{k-1}}, \vct{v}\rangle + \gamma\mu\langle \vct{e_{k-1}}, \vct{a_{s}} \rangle- \langle \vct{e_{k-1}}, \vct{v}\rangle|^2\\
&= |\mu^2\langle\vct{e_{k-1}}, \vct{v}\rangle - \gamma\mu\langle \vct{e_{k-1}}, \vct{a_{s}} \rangle|^2.
\end{align*}
This gives
\begin{align*}
\|\vct{x} - \vct{z'}\|_2^2 
&= \|\vct{e_{k-1}}\|_2^2 - |\langle \vct{e_{k-1}}, \vct{a_{s}}\rangle|^2 - |\langle \vct{e_{k-1}}, \vct{v}\rangle|^2 +|\mu^2\langle\vct{e_{k-1}}, \vct{v}\rangle - \gamma\mu\langle \vct{e_{k-1}}, \vct{a_{s}} \rangle|^2.
\end{align*}

Combining this identity with~\eqref{eq:new}, we now relate the expected error in the two-subspace Kaczmarz algorithm, $\mathbb{E}\|\vct{x} - \vct{x_{k}}\|_2^2$ to the expected error of the standard method, $\mathbb{E}\|\vct{x} - \vct{z'}\|_2^2$ as follows:
\begin{equation}\label{eq:stand}
\mathbb{E}\|\vct{x} - \vct{x_{k}}\|_2^2= \mathbb{E}\|\vct{x} - \vct{z'}\|_2^2 - \mathbb{E}|\mu^2\langle\vct{e_{k-1}}, \vct{v}\rangle - \gamma\mu\langle \vct{e_{k-1}}, \vct{a_{s}} \rangle|^2.
\end{equation}

It thus remains to analyze the last term. Since we select the two rows $r$ and $s$ independently from the uniform distribution over pairs of distinct rows, the expected error is just the average of the error over all $m^2 - m$ ordered choices ${r,s}$.  To this end we introduce the notation $\mu_{r,s} = \langle \vct{a_{r}}, \vct{a_{s}}\rangle$.  Then by definitions of $\vct{v}$, $\mu$ and $\gamma$,  
\begin{align*}
\mathbb{E}|\mu^2\langle\vct{e_{k-1}}&, \vct{v}\rangle - \gamma\mu\langle \vct{e_{k-1}}, \vct{a_{s}} \rangle|^2\\
&= \frac{1}{m^2-m}\sum_{r\ne s} \left|\frac{\mu_{r,s}^2}{\sqrt{1-\mu_{r,s}^2}} (\langle \vct{e_{k-1}}, \vct{a_{r}}\rangle - \mu_{r,s}\langle \vct{e_{k-1}}, \vct{a_{s}}\rangle) - \mu_{r,s}\sqrt{1 - \mu_{r,s}^2}\langle \vct{e_{k-1}}, \vct{a_{s}}\rangle  \right|^2\\
&= \frac{1}{m^2-m}\sum_{r\ne s} \left|\frac{\mu_{r,s}^2}{\sqrt{1-\mu_{r,s}^2}} \langle \vct{e_{k-1}}, \vct{a_{r}}\rangle - \left(\frac{\mu_{r,s}^3}{\sqrt{1-\mu_{r,s}^2}} + \mu_{r,s}\sqrt{1 - \mu_{r,s}^2} \right)\langle \vct{e_{k-1}}, \vct{a_{s}}\rangle\right|^2\\
&= \frac{1}{m^2-m}\sum_{r\ne s} \left|\frac{\mu_{r,s}^2}{\sqrt{1-\mu_{r,s}^2}} \langle \vct{e_{k-1}}, \vct{a_{r}}\rangle - \left(\frac{\mu_{r,s}}{\sqrt{1-\mu_{r,s}^2}} \right)\langle \vct{e_{k-1}}, \vct{a_{s}}\rangle\right|^2.\\
\end{align*}

We now recall that for any $\theta,\pi,u,$ and $v$,
$$
(\theta u - \pi v)^2 + (\theta v - \pi u)^2 \geq (|\pi| - |\theta|)^2(u^2+v^2). 
$$
  
  Setting $\theta_{r,s} = \frac{\mu_{r,s}^2}{\sqrt{1-\mu_{r,s}^2}}$ and $\pi_{r,s} = \frac{\mu_{r,s}}{\sqrt{1-\mu_{r,s}^2}} $, we have by rearranging terms in the symmetric sum,

\begin{align}\label{eq:bound}
\mathbb{E}|\mu^2\langle\vct{e_{k-1}}&, \vct{\theta}\rangle - \gamma\mu\langle \vct{e_{k-1}}, \vct{a_{s}} \rangle|^2\notag\\
&= \frac{1}{m^2-m}\sum_{r\ne s} \left|\theta_{r,s} \langle \vct{e_{k-1}}, \vct{a_{r}}\rangle - \pi_{r,s}\langle \vct{e_{k-1}}, \vct{a_{s}}\rangle\right|^2\notag\\
&= \frac{1}{m^2-m}\sum_{r < s} \left|\theta_{r,s} \langle \vct{e_{k-1}}, \vct{a_{r}}\rangle - \pi_{r,s}\langle \vct{e_{k-1}}, \vct{a_{s}}\rangle\right|^2 + \left|\theta_{r,s} \langle \vct{e_{k-1}}, \vct{a_{s}}\rangle - \pi_{r,s}\langle \vct{e_{k-1}}, \vct{a_{r}}\rangle\right|^2\notag\\
&\geq \frac{1}{m^2-m}\sum_{r < s} (|\pi_{r,s}| - |\theta_{r,s}|)^2\left(\langle \vct{e_{k-1}}, \vct{a_{r}}\rangle^2 + \langle \vct{e_{k-1}}, \vct{a_{s}}\rangle^2\right) \notag\\
&= \frac{1}{m^2-m}\sum_{r < s} \Big( \frac{|\mu_{r,s}|-\mu_{r,s}^2}{\sqrt{1-\mu_{r,s}^2}} \Big)^2\left( \langle \vct{e_{k-1}}, \vct{a_{r}} \rangle^2 + \langle \vct{e_{k-1}}, \vct{a_{s}}\rangle^2 \right).
\end{align}

Since selecting two rows without replacement (i.e. guaranteeing not to select the same row back to back) can only speed the convergence, we have from~\eqref{SV} that the error from the standard randomized Kaczmarz method satisfies
$$
\E\|\vct{x}-\vct{z}'\|_2^2 \leq (1 - 1/R)^2\|\vct{x} - \vct{x_{k-1}}\|_2^2.
$$

Combining this with~\eqref{eq:stand} and~\eqref{eq:bound} yields the desired result.

\end{proof}

Although the result of Lemma~\ref{lem:main} is tighter, the coherence parameters $\delta$ and $\Delta$ of~\eqref{mus} allow us to present the following result which is not as strong but simpler to state.

\begin{lemma}\label{lem:simple}
 Let $\vct{x_{k}}$ denote the estimation to $\mtx{A}\vct{x} = \vct{b}$ in the $k$th iteration of the two-subspace Kaczmarz method.  Denote the rows of $\mtx{A}$ by $\vct{a}_1, \vct{a}_2, \ldots \vct{a}_m$.  Then 
$$
\mathbb{E}\|\vct{x} - \vct{x_{k}}\|_2^2 \leq \left(\left(1 - \frac{1}{R}\right)^2 - \frac{D}{R} \right)\|\vct{x} - \vct{x_{k-1}}\|_2^2 ,
$$

where $D = \min\Big\{ \frac{\delta^2(1-\delta)}{1+\delta}, \frac{\Delta^2(1-\Delta)}{1+\Delta} \Big\}$, $\delta$ and $\Delta$ are the coherence parameters as in~\eqref{mus}, and $R = \|\mtx{A}^{-1}\|^2\|\mtx{A}\|_F^2$ denotes the scaled condition number.
\end{lemma}
\begin{proof}
By Lemma~\ref{lem:main} we have
\begin{equation}\label{eq:lemma}
\nonumber
\mathbb{E}\|\vct{x} - \vct{x_{k}}\|_2^2 \leq \left(1 - \frac{1}{R}\right)^2\|\vct{x} - \vct{x_{k-1}}\|_2^2 - \frac{1}{m^2-m}\sum_{r < s} C_{r,s}^2\left(\langle \vct{x} - \vct{x_{k-1}}, \vct{a_{r}}\rangle^2 + \langle \vct{x} - \vct{x_{k-1}}, \vct{a_{s}}\rangle^2\right),
\end{equation}
where 
$$
C_{r,s} = \frac{|\langle\vct{a_{r}}, \vct{a_{s}}\rangle|-\langle\vct{a_{r}}, \vct{a_{s}}\rangle^2}{\sqrt{1-\langle\vct{a_{r}}, \vct{a_{s}}\rangle^2}}.
$$

By the assumption that $\delta \leq |\langle \vct{a_{r}}, \vct{a_{s}} \rangle| \leq \Delta$, we have 
$$
C_{r,s}^2 \geq \min\Big\{ \frac{\delta^2(1-\delta)}{1+\delta}, \frac{\Delta^2(1-\Delta)}{1+\Delta} \Big\} = D.
$$

Thus we have that 
\begin{align}\label{eq:final}
\frac{1}{m^2-m}\sum_{r < s} &C_{r,s}^2\left(\langle \vct{x} - \vct{x_{k-1}}, \vct{a_{r}}\rangle^2 + \langle \vct{x} - \vct{x_{k-1}}, \vct{a_{s}}\rangle^2\right)\notag\\
&\geq \frac{D}{m^2-m}\sum_{r < s} \left(\langle \vct{x} - \vct{x_{k-1}}, \vct{a_{r}}\rangle^2 + \langle \vct{x} - \vct{x_{k-1}}, \vct{a_{s}}\rangle^2\right)\notag\\
&= \frac{D(m-1)}{m^2-m}\sum_{r=1}^m \langle \vct{x} - \vct{x_{k-1}}, \vct{a_{r}}\rangle^2\notag\\
&\geq \frac{D}{m}\cdot\frac{\|\vct{x} - \vct{x_{k-1}}\|_2^2}{\|\mtx{A}^{-1}\|_2^2}.
\end{align}

In the last inequality we have employed the fact that for any $\vct{z}$,
$$
\sum_{r=1}^m \langle \vct{z}, \vct{a_{r}}\rangle^2 \geq \frac{\|\vct{z}\|_2^2}{\|\mtx{A}^{-1}\|_2^2}.
$$

Combining~\eqref{eq:final} and~\eqref{eq:lemma} along with the definition of $R$ yields the claim.

\end{proof}

Applying Lemma~\ref{lem:simple} recursively and using the fact that the selection of rows in each iteration is independent yields our main result Theorem~\ref{thm:main}.

\section{Conclusion}\label{sec:conclusion}

As is evident from Theorems~\ref{thm:main}, the two-subspace Kaczmarz method provides exponential convergence in expectation to the solution of $\mtx{A}\vct{x} = \vct{b}$.  The constant in the rate of convergence for the two-subspace Kaczmarz method is at most equal to that of the best known results for the randomized Kaczmarz method~\eqref{SV}.  When the matrix $\mtx{A}$ has many correlated rows, the constant is significantly lower than that of the standard method, yielding substantially faster convergence.  This has positive implications for many applications such as nonuniform sampling in Fourier analysis, as discussed in Section~\ref{sec:intro}. 

We emphasize that the bounds presented in our main theorems are weaker than what we actually prove, and that even when $\delta$ is small, if the rows of $\mtx{A}$ have many correlations, Lemma~\ref{lem:main} still guarantees improved convergence. For example, if the matrix $\mtx{A}$ has correlated rows but contains a pair of identical rows and a pair of orthogonal rows, it will of course be that $\delta = 0$ and $\Delta = 1$.  However, we see from the lemmas in the proofs of our main theorems that the two-subspace method still guarantees substantial improvement over the standard method.  Numerical experiments in cases like this produce results identical to those in Section~\ref{sec:numerics}. 

It is clear both from the numerical experiments and Theorem~\ref{thm:main} that the two-subspace Kaczmarz performs best when the correlations $\<\vct{a_r}, \vct{a_s}\>$ are bounded away from zero.  In particular, the two-subspace method offers the most improvement over the standard method when $\delta$ is large.  The dependence on $\Delta$, however, is not as straightforward.  Theorem~\ref{thm:main} suggests that when $\Delta$ is very close to $1$ the two-subspace method should provide similar convergence to the standard method.  However, in the experiments of Section~\ref{sec:numerics} we see that even when $\Delta \approx 1$, the two-subspace method still outperforms the standard method.  This exact dependence on $\Delta$ appears to be only an artifact of the proof. 

\subsection{Extensions to noisy systems and higher subspaces}
As is the case for many iterative algorithms, the presence of noise introduces complications both theoretically and empirically.  We show in~\cite{NW12:2srkTech} that with noise the two-subspace method provides expected exponential convergence to a noise threshold proportional to the largest entry of the noise vector $\vct{w}$.    
A further and important complication that noise introduces is \emph{semi-convergence}, a well-known effect in Algebraic Reconstruction Technique (ART) methods (see e.g.~\cite{ENP10:semi}).  It remains an open problem to determine the optimal stopping condition without knowledge of the solution $\vct{x}$.  See~\cite{NW12:2srkTech} for more details.  Alternatively, the optimal trade-off between speed and accuracy may be reached by employing a hybrid Kaczmarz algorithm which initially implements two-subspace Kaczmarz iterations to reach an approximate solution quickly, but switches to standard Kaczmarz iterations after a certain number of iterations to arrive at a more accurate final approximation.  

Finally, a natural extension to our method would be to use more than two rows in each iteration.  Indeed, extensions of the two-subspace algorithm to arbitrary subspaces can be analyzed~\cite{NT12:paving}.

\bibliography{rk}

\end{document}